\documentclass{article}
\usepackage[utf8]{inputenc}
\usepackage[english]{babel}
\usepackage{amsmath}
\usepackage{amsfonts}
\usepackage{amssymb}
\usepackage{amsthm}
\usepackage{graphicx}
\usepackage{array}
\usepackage[capitalize]{cleveref}
\usepackage{enumitem}

\newcommand{\footremember}[2]{
    \footnote{#2}
    \newcounter{#1}
    \setcounter{#1}{\value{footnote}}
}
\newcommand{\footrecall}[1]{
    \footnotemark[\value{#1}]
} 

\newcommand{\overbar}[1]{\mkern 1.5mu\overline{\mkern-1.5mu#1\mkern-1.5mu}\mkern 1.5mu}
\newcommand{\setdelim}{\; | \;}
\newcommand{\lspace}{{L^2(X,m)}}

\newcommand{\E}{\mathcal{E}}
\newcommand{\F}{\mathcal{F}}
\newcommand{\D}{\mathfrak{D}}

\crefname{property}{Property}{Properties}
\crefname{assum}{Assumption}{Assumptions}

\theoremstyle{plain}
\newtheorem{thm}{Theorem}[section]
\newtheorem{lem}[thm]{Lemma}
\newtheorem{corollary}[thm]{Corollary}

\theoremstyle{definition}
\newtheorem{defn}[thm]{Definition} 
\newtheorem{exmp}[thm]{Example} 
\newtheorem{remark}[thm]{Remark}

\DeclareMathOperator{\normcap}{Cap_{\mathfrak{D}}}

\DeclareMathOperator{\inftynorm}{\Vert \cdot \Vert_\infty}
\DeclareMathOperator{\supp}{supp}

\DeclareMathOperator{\dom}{dom}

\begin{document}

\author{Ralph Chill \footremember{TUD}{Institut f\"ur Analysis, Fakult\"at Mathematik, TU Dresden, 01062 Dresden, Germany, \texttt{ralph.chill@tu-dresden.de}, \texttt{burkhardclaus@gmx.de}} \and Burkhard Claus \footrecall{TUD}}

\date{\today}

\title{Domination of nonlinear semigroups generated by regular, local Dirichlet forms}

\maketitle
 \begin{abstract}
     In this article we study perturbations of local, nonlinear Dirichlet forms on arbitrary topological measure spaces. We show that the semigroup of a local Dirichlet form \(\E\) dominates the semigroup generated by another functional \(\mathcal{F}\) if, and only if, \(\mathcal{F}\) is a specific zero order perturbation of \(\E\). This helps to classify the perturbations that lie between Neumann and Dirichlet boundary conditions.
 \end{abstract}

\section{Introduction}

In \cite{AW_Laplace_what_is_in_between} Arendt and Warma showed the following result: Let \(\big(e^{t \Delta_N} \big)_{t\geq 0}\) and \(\big(e^{t \Delta_D} \big)_{t\geq 0}\) be the heat semigroups generated by the Laplacian with Neumann and Dirichlet boundary conditions, respectively, on some regular domain \(\Omega \subseteq \mathbb{R}^n\) and let \(S\) be another linear \(C_0\)-semigroup generated by a (bilinear, symmetric) Dirichlet form. If \(S\) is dominated by \(\big(e^{t \Delta_N} \big)_{t\geq 0}\) and in turn dominates \(\big(e^{t \Delta_D} \big)_{t\geq 0}\) in the sense that
     \begin{align} \label{eqn:intro_1}
     \vert S_t f\vert \leq  e^{t \Delta_N} \vert f \vert\\
     \vert e^{t \Delta_D} f\vert \leq  S_t \vert f \vert \nonumber
\end{align}
for every \(f \in L^2(\Omega)\) and every \(t\geq 0\), then there is a measure \(\mu\) on \(\partial \Omega\) such that \(S\) is generated by the closure of the form
\begin{align*}
    a(u,v)=\int_\Omega \nabla u \nabla v \; \mathrm{d}x+ \int_{\partial \Omega} u v \; \mathrm{d}\mu
\end{align*}
for \(u,v \in L^2(\Omega) \cap C^\infty(\overbar{\Omega})\). The representation of the form \(a\) shows that \(S\) is a semigroup generated by the Laplacian with Robin boundary conditions. The converse statement, namely that the heat semigroup generated by the Laplacian with Robin boundary conditions corresponding to any admissible measure \(\mu\) is always sandwiched in between the semigroups generated by the Laplacian with Neumann and Dirichlet boundary conditions in the sense of \eqref{eqn:intro_1}, is easy to see. Hence, this result fully characterises the sandwiched semigroups. To be precise, let us mention that he original statement in \cite{AW_Laplace_what_is_in_between} actually assumed that the form generating \(S\) is local, but Akhlil \cite{Ak18} shows that locality follows automatically from the positivity of \(S\) which in turn follows from the fact that \(S\) dominates \(\big(e^{t \Delta_D} \big)_{t\geq 0}\). In \cite{CW_p_Laplace_what_is_in_between} an analogous theorem characterizing sandwiched semigroups was shown for the \(p\)-Laplace operator and its associated semigroups. The main aim of this paper is to show a similar result for a large class of nonlinear Dirichlet forms as defined in \cite{CG_Nonlinear_Dirichlet_Forms,Cl23}. Both the Laplace and the \(p\)-Laplace operator are examples of such nonlinear Dirichlet forms.

In \cite{Cl23}, the second author introduced Dirichlet spaces and capacities associated to nonlinear Dirichlet forms \(\E\) on \(\lspace\), where \((X,m)\) is some topological Hausdorff measure space. We give a short recollection of these results in the next section. In the following part we define the abstract boundary of \(X\) associated with \(\E\). This enables us to investigate boundary conditions even in abstract settings where no topological boundary of \(X\) is available. We then prove a nonlinear version of the Riesz-Markov theorem, which states that any local monotone functional on the Dirichlet space \(\D\) is given by some integral. In the last part we use this representation theorem to show \cref{main theo}, the main result of this article, which shows, that, under additional assumptions, a semigroup \(S\) is dominated by the semigroup generated by \(\E\) if and only if \(S\) is generated by some specific integral perturbation of \(\E\). This yields a corollary in the spirit of \cite{AW_Laplace_what_is_in_between,CW_p_Laplace_what_is_in_between}.

\section{Preliminaries}

We cite some definitions and results from \cite{Cl23}. In the following, \((X,m)\) denotes a Hausdorff topological measure space such that \(\supp(m)=X\).

\begin{defn} \label{def:dirichlet_form}
Let \(\mathcal{E} : \lspace \rightarrow [0,\infty]\) be a convex and lower semicontinuous functional with dense effective domain. We call \(\mathcal{E}\) a \underline{Dirichlet form} if
\begin{align}\label[property]{diricheqn1}
&\mathcal{E}(u \wedge v) + \mathcal{E}(u \vee v) \leq \mathcal{E}(u)+\mathcal{E}(v)
\end{align}
and
\begin{align} 
\nonumber &\mathcal{E}\bigg( v+\frac{1}{2}\big( (u-v+\alpha)_+-(u-v-\alpha)_- \big) \bigg) \\
 & \quad \quad +\mathcal{E}\bigg( u-\frac{1}{2}\big( (u-v+\alpha)_+-(u-v-\alpha)_- \big) \bigg) \leq \mathcal{E}(u)+\mathcal{E}(v) \label[property]{diricheqn2}
\end{align}
for every \(u,v \in \lspace,\alpha>0\).

We call \(\mathcal{E}\) \underline{symmetric} if \(\E(0)=0\) and \(\E(-u)=\E(u)\) for all \(u \in \lspace\).
\end{defn}

Recall that every convex, lower semicontinuous functional \(\E\) generates a semigroup of contractions on \(\lspace\), see \cite[Th\'eor\`emes III.3.1, III.3.2]{Brezis_subgradients} (more precisely, the negative subgradient of \(\E\) generates the semigroup), and \(\E\) is a Dirichlet form if, and only if, the associated semigroup is order preserving and \(L^\infty\)-contractive. In fact, \(S\) is order preserving if and only if property \eqref{diricheqn1} holds \cite[Corollaire 2.2]{Barthelemy_invariant_convex_sets}, \cite[Theorem 3.8]{CG_Nonlinear_Dirichlet_Forms}, and it is \(L^\infty\)-contractive if and only if property \eqref{diricheqn2} holds \cite[Corollary 3.9]{CG_Nonlinear_Dirichlet_Forms}. 

Furthermore, Barth\'elemy \cite{Barthelemy_invariant_convex_sets} showed the following result:

\begin{thm}\label{thm_characterisation_domination_and_abseqn}
Let \(\E,\mathcal{F} : \lspace \rightarrow [0,\infty]\) be two Dirichlet forms and \(T,S\) the associated semigroups. Then \(S\) is dominated by \(T\) in the sense that
 \begin{align*}
     \vert S_t f\vert \leq  T_t \vert f \vert
 \end{align*}
 for every \(f \in \lspace\) and \(t\geq 0\), if, and only if,
 \begin{align*}
 \mathcal{F}( ( \vert u \vert \wedge v ) \textup{sgn}(u)) + \mathcal{E}( \vert u \vert \vee v) \leq \mathcal{F}(u)+  \mathcal{E}(v)
 \end{align*}
 for every \(u,v \in \lspace\), \(v \geq 0\).
 \end{thm}

 As a corollary this yields
 \begin{align*}
     \E(\vert u \vert) \leq \E(u)
 \end{align*}
 if \(\E\) is a symmetric Dirichlet form.

Set
\begin{align*}
    \mathcal{E}_1(u)= \Vert u \Vert_\lspace^2 + \mathcal{E}(u)
\end{align*}
and
\begin{align*}
\mathfrak{D}=\{ u\in \lspace \setdelim \exists \lambda>0: \mathcal{E}_1(\lambda u)< \infty \}
\end{align*}
together with
\begin{align*}
\Vert u \Vert_\mathfrak{D} = \inf\Big\{ \lambda>0 : \mathcal{E}_1\Big(\frac{u}{\lambda} \Big)\leq 1 \Big\}.
\end{align*}

\begin{thm}\label{Thm:D_Riesz_Subspace_Norm}
Let \(\mathcal{E}\) be a symmetric Dirichlet form. Then \((\D, \Vert \cdot \Vert_\mathfrak{D})\) is a Banach space and a lattice under the pointwise order such that
\begin{align*}
    \Vert f \wedge g\Vert_\D \leq \Vert f \Vert_\D + \Vert g \Vert_\D
\end{align*}
and 
\begin{align*}
    \Vert -c \vee f \wedge c\Vert_\D \leq \Vert f \Vert_\D
\end{align*}
for all \(f,g \in \D\) and \(c\geq 0\). \label{thm:continuous_lattice_operations}
If, in addition, \(\dom \E = \D\), then the lattice operations \(\wedge,\vee\) are continuous and \(-c \vee u \wedge c\) converges to \(u\) for \(c \rightarrow \infty\) and to \(0\) for \(c \rightarrow 0\).
\end{thm}

In the following \(\D_b= \D \cap L^\infty(X,m)\) denotes the space of bounded functions in \(\D\).

\begin{defn}
Let \(A \subseteq X \). We define the set
\begin{align*}
    \mathcal{L}_A=\{u \in \lspace \setdelim u\geq 1 \text{ on } U, A \subseteq U,\; U \text{ open} \},
\end{align*}
and the \underline{norm-capacity} by
\[
\normcap(A)=\inf\{\Vert u \Vert_\D \setdelim u \in \mathcal{L}_A\} .
\]
\end{defn}

\begin{lem}
 Let \(A,B \subseteq X\). Then \(\normcap(A \cup B) \leq \normcap(A)+\normcap(B)\), and if \(A\subseteq B\), then \(\normcap(A) \leq \normcap(B)\).
\end{lem}

\begin{lem}
Let \(A_n ,A \subseteq X\) such that \(\bigcup_{n=1}^\infty A_n =A\). Then \(\sum_{n =1}^\infty \normcap(A_n)\geq \normcap(A)\).
\end{lem}

\begin{lem}
Let \(K_n \subseteq X\) be compact subsets such that \(K_n \downarrow K \). Then
\begin{align*}
    \normcap(K)= \inf_{n \in \mathbb{N}} \normcap(K_n).
\end{align*}
\end{lem}

\begin{defn}
    We call a set \(A\subseteq X\) \underline{polar}, if \( \normcap(A)=0\) and some property depending on \(x\in X\) holds \underline{quasi everywhere} (q.e.), if there is a polar set \(A\) such that the property holds everywhere on \(X\setminus A\).
 We say \(f: X  \rightarrow Y\) for some topological space \(Y\) is \underline{quasicontinuous}, if for every \(\epsilon > 0 \) there is an open set \(O\subseteq X \) such that \(\normcap(O) \leq \epsilon \) and \(f|_{O^c}\) is continuous.

Additionally we call \(f \in \lspace\) \underline{quasicontinuous}, if there is a representative which is quasicontinuous. Whenever this is the case, we denote this representative again by \(f\).
\end{defn}

\begin{lem}\label{lem:polar_implies_nullset}
 Let \( A \subseteq X\) be measurable and polar. Then \(m(A)=0\).
\end{lem}

\begin{thm} \label{aeqe}
Let \(U \subseteq X\) be an open subset and \(f:X \rightarrow \mathbb{R} \) quasicontinuous. Then
\begin{align*}
f \geq 0 \text{ a.e. on }U \Longleftrightarrow f \geq 0 \text{ q.e. on }U.
\end{align*}
\end{thm}

\begin{corollary}\label{thm:unique_qc_representatives}
Let \(f_1 , f_2:X \rightarrow \mathbb{R}\) be two quasicontinuous representatives of some \(f \in \lspace\). Then \(f_1=f_2\) quasi everywhere on \(X\).
\end{corollary}

\begin{thm}
Let \(f \in \overbar{\mathfrak{D} \cap C(X)}\). Then \(f\) is quasicontinuous on \(X\).
\end{thm}

\begin{thm}\label{thm:D_convergnes_implies_pointwise_quasi_everywhere}
Let \((f_n)_n\) be a sequence of quasicontinuous functions in \(\mathfrak{D}\) and \(f \in \mathfrak{D}\) with \(f_n \rightarrow f \) in \( \mathfrak{D}\). Then \(f\) is quasicontinuous and there exists a subsequence which converges pointwise quasi everywhere.
\end{thm}

All proofs can be found in \cite{Cl23}.

\section{The abstract boundary}

A major point of interest in the study of partial differential equations are boundary conditions. In the next chapters we want to study these conditions in more detail. The most important question in this context is: What is the boundary of our topological space \(X\)? In general this is not well defined. In this section we construct a compact Hausdorff space \(\hat{X}\) with a measure \(\hat{m}\) on it such that \(\lspace\) is isometric to a subspace of \(L^2 (\hat{X},\hat{m})\), such that \(\E\) can be seen as a Dirichlet form on \(L^2 (\hat{X},\hat{m})\), and such that the Dirichlet space \(\mathfrak{D}\) associated with \(\E\) is isometric to \(\mathfrak{D}\). Then we define the boundary of \(X\) as \(\partial X := \hat{X} \setminus X\).

Again we fix a symmetric Dirichlet form \(\mathcal{E}\).

\begin{defn}
We call a closed unital subalgebra \(\hat{\mathcal{A}} \subseteq C_b(X)\) a \underline{core} of \(\mathcal{E}\) if \(\mathfrak{D} \cap \hat{\mathcal{A}}\) is dense in \(\mathfrak{D}\) and \(\hat{\mathcal{A}}\) is a lattice. We call \(\hat{\mathcal{A}}\) \underline{regular} if 
\begin{align*}
    \overbar{\langle \mathfrak{D} \cap \hat{\mathcal{A}}, \{1\} \rangle}^{\inftynorm}=\langle \overbar{\mathfrak{D} \cap \hat{\mathcal{A}}}^{\inftynorm}, \{1\} \rangle= \hat{\mathcal{A}}
\end{align*}
and we call \(\mathcal{E}\) \underline{regular}, if it possesses a regular core.

We call a subalgebra \(\hat{\mathcal{A}}\subseteq C_b(X)\) a \underline{precore} if \(\hat{\mathcal{A}}\) is a lattice and \(\mathfrak{D} \cap \hat{\mathcal{A}}\) is dense in \(\mathfrak{D}\). We call \(\hat{\mathcal{A}}\) \underline{regular} if the closure of \(\langle \hat{\mathcal{A}} , 1 \rangle\) is regular. All following definitions also work for precores by using the associated core.
\end{defn}
The following is a real version of the well known Gelfand-Naimark theorem.

\begin{thm}[Arens 1947] \label{arens}
Let \(\hat{\mathcal{A}}\) be a real commutative Banach algebra with unit e such that
\( \Vert e \Vert = 1\). Then there is a compact Hausdorff space \(\hat{X}\) such that \(\hat{\mathcal{A}}\) is isometrically isomorphic to the algebra \(C(\hat{X})\), if and only if 
\begin{align} \label{eq.arens}
 \Vert a \Vert ^2 \leq \Vert a^2 +b^2 \Vert \text{ for all } a,b \in \hat{\mathcal{A}},  
\end{align}
and in this case \(\hat{X}\) is the compact Hausdorff space of all characters on \(\hat{\mathcal{A}}\) and the isomorphism is given by the Gelfand transform.
\end{thm}

See \cite{Arens_Representation_of_star_algebras} for a proof.
 
We observe, that every closed unital subalgebra \(\hat{\mathcal{A}} \subseteq C_b(X)\) fulfills the property \eqref{eq.arens} from \cref{arens}, since, for \(a,b \in C_b(X)\), \((b(x))^2\geq 0\) for every \(x \in X\) and hence,
\begin{align*}
    \Vert a \Vert_\infty ^2= \Vert a^2 \Vert_\infty = \sup_{x\in X} \vert a^2(x)\vert \leq \sup_{x\in X} \vert a^2(x)+b^2(x)\vert=\Vert a^2 +b^2 \Vert_\infty.
\end{align*}

From now on we assume that \(\mathcal{E}\) is a symmetric Dirichlet form which possesses a core \(\hat{\mathcal{A}}\). We write \(\hat{X}\) for the compact Hausdorff space given by \cref{arens} such that \(\hat{\mathcal{A}}= C(\hat{X})\). Additionally, let
\begin{align*}
    \omega:  X &\rightarrow \hat{X}, \\
     x & \mapsto \delta_x,
\end{align*}
be the canonical embedding and denote by
\begin{align*}
\Delta : & \hat{\mathcal{A}} \rightarrow C(\hat{X}), \\
& f \mapsto \hat{f},
\end{align*} 
the Gelfand transform. The map \(\omega\) is neither injective nor surjective in general, and it is injective if and only if \(\hat{\mathcal{A}}\) separates the points of \(X\). In addition, \( \omega(X)\) is dense in \(\hat{X}\) \cite{GRS_commutative_normed_rings} and 
\begin{align*}
    \omega^* :   C(\hat{X}) &\rightarrow \hat{\mathcal{A}} \\
     f &\mapsto f \circ \omega
\end{align*}
is the inverse of the Gelfand transform. We define the push forward measure \(\hat{m}\) on \(\hat{X}\) by 
\begin{align*}
    \hat{m}(A)=\omega^*m(A)=m(\omega^{-1}(A)) \text{ for } A\in \mathcal{B}(\hat{X}) .
\end{align*}
Then \(\hat{m}\) is a Borel measure.
Let \(f \in L^2(\hat{X},\hat{m})\). By definition of the push forward measure, 
\begin{align*}
    \int_X (\omega^*f)^2 dm= \int_X (f \circ \omega)^2 dm= \int_{\hat{X}} f^2 d\hat{m}.
\end{align*}
Hence, the following lemma holds:

\begin{lem}
The map \(\omega^*: L^2(\hat{X},\hat{m}) \rightarrow \lspace \) is an isometric injection, which preserves the pointwise order.
\end{lem}

Furthermore let us define the functional \(\hat{\mathcal{E}}\) on \(L^2(\hat{X},\hat{m})\) by
\begin{align*}
    \hat{\mathcal{E}}(f)=\omega^*\mathcal{E}(f)=\mathcal{E}(f\circ \omega) \text{ for } f \in L^2(\hat{X},\hat{m}).
\end{align*}
This is a Dirichlet form on \(L^2(\hat{X},\hat{m})\). We denote the associated Dirichlet space by \( (\hat{\mathfrak{D}},{\Vert \cdot \Vert}_{\hat \D})\).

\begin{lem}
For every \(f \in \hat{\D}\)
\begin{align*}
    {\Vert f \Vert}_{\hat\D} =\Vert \omega^*f \Vert_\D=\Vert f\circ \omega \Vert_\D
\end{align*}
and hence \( \omega^*: \hat{\mathfrak{D}}\rightarrow \mathfrak{D}\) is an isometric injection.
\end{lem}

Indeed, even the following is true.

\begin{thm} \label{dirichlet.spaces}
Let \(\mathcal{E}\) be a regular, symmetric Dirichlet form with core \(\hat{\mathcal{A}}\). Then the map \( \omega^*: \hat{\mathfrak{D}}\rightarrow \mathfrak{D}\) is an isometric isomorphism.
\end{thm}

\begin{proof}
We have to show surjectivity. We already know, that \(\omega^*\) maps \(\hat{\mathfrak{D}}\) to \(\mathfrak{D}\) and \(C(\hat{X})\) bijectively to \(\hat{\mathcal{A}}\). Hence, it maps \(\hat{\mathfrak{D}} \cap C(\hat{X})\) bijectively to \(\mathfrak{D} \cap \hat{\mathcal{A}}\) and the inverse is given by \(\Delta\). Since \(\hat{\mathcal{A}}\) is a core, \({\mathfrak{D}} \cap C(\hat{X})\) is dense in \(\mathfrak{D}\). Thus \(\Delta\) is an isometric bijection defined on a dense subset of \(\mathfrak{D}\). Let us extend \(\Delta\) to \(\mathfrak{D}\). For an arbitrary \(f\in \mathfrak{D}\), there is a sequence \((f_n)_n\) in \(\mathfrak{D} \cap \hat{\mathcal{A}}\) converging to \(f\). Therefore,
\begin{align*}
    \omega^*(\Delta(f))= \lim_{n\rightarrow \infty}\omega^*(\Delta(f_n))=
    \lim_{n\rightarrow \infty}f_n=f.
\end{align*}
This implies that \(\omega^*\) has a right inverse and thereby \(\omega^*\) is surjective and its inverse is given by \(\Delta\).
\end{proof}

If the algebra \(\hat{\mathcal{A}}\) separates the points of \(X\), then \(\hat{X}\) is a compactification of \(X\). Clearly, this compactification depends on the algebra  \(\hat{\mathcal{A}}\), and in order to keep the compactification as small as possible it is desirable to keep the algebra as small as possible. We call \(\partial X := \hat{X}\setminus X\) the \underline{boundary} of \(X\). By \cref{dirichlet.spaces}, every function in the Dirichlet space \(\mathfrak{D}\) can be uniquely "extended" to a function in \(\hat{\mathfrak{D}}\), that is, to a function which is defined on the compactification \(\hat{X}\). Also, the capacity associated with the form \(\hat{\E}\) is a capacity on subsets of \(\hat{X}\) which extends the capacity on \(X\). Every function \(f\in\hat{\mathfrak{D}}\) is quasi-continuous on \(\hat{X}\), and in particular it is quasi-continuous on the boundary \(\partial X\). 

\begin{exmp}
Let \(X\) be a bounded \(C^1\)-domain in \(\mathbb{R}^d\) and
\begin{align*}
    \mathcal{E}(f)= \begin{cases}
                        \frac{1}{2}\Vert f \Vert_{H^1(X)}^2 &\text{ if } f\in H^1(X) , \\
                        \infty &\text{ otherwise}.
                    \end{cases}
\end{align*}
Then \(\mathfrak{D}= H^1(X)\) and \(\hat{\mathcal{A}}= C(\bar{X})\), where \(\bar{X}\) is the usual closure of \(X\) in the Euclidean space \(\mathbb{R}^d\), is a regular core of \(\mathcal{E}\) (here, we need some regularity of the usual boundary; the assumed \(C^1\)-regularity of the boundary is sufficient). Since \(\hat{\mathcal{A}}\) separates points, the map \( \omega\) is injective, and since \(\bar{X}\) is already compact, we have \(\hat{X} = \bar{X}\) and the boundary \(\partial X\) is the usual boundary in the Euclidean space. 
\end{exmp}

\section{Separation of points}

\begin{thm}\label{septheo1}
Let \(\mathcal{E}\) be a regular Dirichlet form, \(\hat{\mathcal{A}}\) a regular core and \(A,B \subseteq \hat{X}\) two disjoint closed subsets. Then there exists a function \(\hat{f}\in \hat{\mathfrak{D}}\cap C(\hat{X})\) such that
\begin{align*}
    &\hat{f}=1 \text{ on } A \text{ and }\; \\
    &\hat{f}=0 \text{ on } B,
\end{align*}
or 
\begin{align*}
    &\hat{f}=0 \text{ on }  A \text{ and }\; \\
    &\hat{f}=1 \text{ on } B.
\end{align*}
\end{thm}

\begin{proof}
Let \(A,B \subseteq \hat{X}\) be two disjoint closed subsets.
Since \(\langle \mathfrak{D} \cap \hat{\mathcal{A}}, \{1\} \rangle\) is dense in \(\hat{\mathcal{A}}\), we know that \(\langle \hat{\mathfrak{D}} \cap C(\hat{X}), \{1\} \rangle\) is dense in \(C(\hat{X})\). On the other hand, \(\hat{X}\) is a compact Hausdorff space and thus normal. By Urysohn's Lemma there is a function \(\hat{h} \in C(\hat{X})\) such that \(0\leq \hat{h}\leq 1\) and 
\begin{align*}
    &\hat{h}=1 \text{ on } A \text{ and }\\
    &\hat{h}=0 \text{ on } B.
\end{align*}
Let \(\epsilon>0\). Then \(\hat{g}=(1+2 \epsilon) \hat{h}-\epsilon\) obeys
\begin{align*}
    &\hat{g}\geq 1+\epsilon \text{ on } A \text{ and }\; \\
    &\hat{g}\leq -\epsilon \;\;\;\text{ on } B ,
\end{align*}
and \( -\epsilon \leq \hat{g} \leq 1+\epsilon\)
By the uniform density mentioned above, we can find a function \(\hat{f}_1 \in \hat{\mathfrak{D}} \cap C(\hat{X})\) and \(\lambda \in \mathbb{R}\) such that 
\begin{align*}
    1+\frac{\epsilon}{2} &< \hat{f}_1 - \lambda < 1+ \frac{3 \epsilon}{2} \text{ on } A \text{ and } \\
    - \frac{3\epsilon}{2}&<\hat{f}_1 - \lambda <- \frac{\epsilon}{2} \;\;\;\;\;\text{ on } B,
\end{align*}
or in other words 
\begin{align*}
     1 + \lambda +\frac{\epsilon}{2} &< \hat{f}_1 < 1 + \lambda +\frac{3\epsilon}{2} \text{ on } A \text{ and } \\
     \lambda -\frac{3\epsilon}{2} &< \hat{f}_1 < \lambda -\frac{\epsilon}{2} \quad \quad \text{ on } B.
\end{align*}
We have to inspect three cases. First, if \(\lambda>0\), we set
\begin{align*}
    \hat{f}= 0 \vee \bigg(\hat{f}_1-\big((-\lambda) \vee \hat{f}_1 \wedge \lambda \big) \bigg) \wedge 1.
\end{align*}
By \cref{Thm:D_Riesz_Subspace_Norm}, \(\hat{f}\in \mathfrak{D}\cap C(K)\) and additionally it satisfies 
\begin{align*}
    &\hat{f}=1 \text{ on } A \text{ and }\; \\
    &\hat{f}=0 \text{ on } B.
\end{align*}

Second, if \(0 \geq \lambda \geq -1 - \frac{\epsilon}{2} \), we take
\begin{align*}
    \hat{f}=\left(\frac{1}{\min_{x\in A}\hat{f}_1(x)}\big(0 \vee \hat{f}_1 \big)\right) \wedge 1,
\end{align*}
where the minimum exists and is greater than \(0\) since \(A\) is compact and \(\hat{f}_1\) continuous. Then
\begin{align*}
    &\hat{f}=1 \text{ on } A\; \text{ and } \\
    &\hat{f}=0 \text{ on } B.
\end{align*}
In the third cases \(\lambda < -1 - \frac{\epsilon}{2} \), we can use
\begin{align*}
    \hat{f}=\left(-\frac{2}{\epsilon}(\hat{f}_1-\hat{f}_1 \vee \lambda)\right)\wedge 1,
\end{align*}
which fulfills 
\begin{align*}
    &\hat{f}=0 \text{ on } A\; \text{ and } \\
    &\hat{f}=1 \text{ on } B.
\end{align*}
This concludes the proof.
\end{proof}

\begin{lem}\label{thm:dirichlet_point}
Let \(\mathcal{E}\) be a regular Dirichlet form and \(\hat{\mathcal{A}}\) its regular core.
Then there is at most one point \(x_0\in \hat{X}\), such that \(\hat{f}(x_0)=0\) for all \(\hat{f} \in \hat{\mathfrak{D}} \cap C(\hat{X}) \).
\end{lem}

\begin{proof}
Let us assume there are two distinct points \(x_1,x_2 \in \hat{X}\) with this property. Since \(\hat{X}\) is Hausdorff, \(\{x_1\}\) and \(\{x_2\}\) are closed sets, and they are clearly disjoint. Hence, the result follows from the previous theorem.
\end{proof}

\begin{exmp}
Let \(X\) be a bounded domain in \(\mathbb{R}^d\) and
\begin{align*}
    \mathcal{E}(f)= \begin{cases}
                        \frac{1}{2}\Vert f \Vert_{H^1(X)}^2 &\text{ if } f\in H^1_0(X) , \\
                        \infty &\text{ otherwise}.
                    \end{cases}
\end{align*}
Then \(\mathfrak{D}= H^1_0(X)\) and \(\hat{\mathcal{A}}=\{ f\in C(\overbar{X}) \setdelim f= \text{constant on } \partial X\}\) is a regular core of \(\mathcal{E}\). Since \(\hat{\mathcal{A}}\) separates points, the map \( \omega\) is injective and we can interprete \(X\) as a subset of \(\hat{X}\). The only character of \(\hat{\mathcal{A}}\), which is not in \(\omega(X)\), is the point evaluation at the boundary. Therefore, we can write \(\hat{X}=\overbar{X} / \partial X\) as a quotient of topological spaces.
Thus, \(\hat{X}\) is the one-point compactification of \(X\). The point \(x_0\) from \cref{thm:dirichlet_point} corresponds to \(\partial X\). Hence, in a certain way we can interpret \(x_0\) as the point, where we enforce Dirichlet boundary conditions.
\end{exmp}

Inspired by this example, we call the unique point \(x_0\) from \cref{thm:dirichlet_point}, if it exists, the \underline{Dirichlet boundary point}.

\begin{thm}\label{septheo2}
Let \(\mathcal{E}\) be a regular Dirichlet form and \(\hat{\mathcal{A}}\) a regular core, \(u \in \hat{\mathfrak{D}}_b\) and \(A,B \subseteq \hat{X}\) two disjoint closed sets. Then there exists a function \(\hat{f}\in \hat{\mathfrak{D}} \cap C(\hat{X})\) such that 
\begin{align*}
    &\hat{f}\geq u \text{ on } A \text{ and } \\
    &\hat{f}=0 \text{ on } B.
\end{align*}
\end{thm}

\begin{proof}
By \cref{septheo1}, there is a function \(\hat{g} \in \hat{\mathfrak{D}}\cap C(K)\) satisfying 
\begin{align*}
    &\hat{g}=1 \text{ on } A \text{ and } \\
    &\hat{g}=0 \text{ on } B,
\end{align*}
or 
\begin{align*}
    &\hat{g}=0 \text{ on } A \text{ and } \\
    &\hat{g}=1 \text{ on } B.
\end{align*}
In the first case one takes \(\hat{f}= \Vert u \Vert_\infty \hat{g}\) and in the second case \( \hat{f}=u- u \wedge \Vert u \Vert_\infty \hat{g}\).
\end{proof}

\begin{thm}\label{septheo3}
Let \(\mathcal{E}\) be a regular Dirichlet form, \(\hat{\mathcal{A}}\) a regular core and \(A,B \subseteq \hat{X}\) two disjoint closed sets. If the Dirichlet boundary point \(x_0 \) either exists and \(x_0 \notin A\) or if \(x_0\) does not exist, then there exists a function \(\hat{f}\in \hat{\mathfrak{D}}\) such that 
\begin{align*}
    &\hat{f}= 1 \text{ on } A \text{ and }\; \\
    &\hat{f}=0 \text{ on } B.
\end{align*}
\end{thm}

\begin{proof}
At first let us assume that \(x_0\) exists. Then we can find \(\hat{g},\hat{h} \in \overbar{\mathfrak{D} \cap C(\hat{X})}^{\inftynorm}\) and \(\lambda \in \mathbb{R}\), such that
\begin{align*}
    &\hat{g}= 1-\lambda \text{ on } A \text{ and }\; \\
    &\hat{g}=0 -\lambda\text{ on } B,
\end{align*}
and 
\begin{align*}
    &\hat{h}= 1 \text{ on } A \cup B \text{ and }\; \\
    &\hat{h}=0 \text{ on } \{x_0\}.
\end{align*}
Hence, 
\begin{align*}
    &\hat{g}+\lambda \hat{h}= 1 \text{ on } A \text{ and }\; \\
    &\hat{g}+\lambda \hat{h}=0 \text{ on } B.
\end{align*}
The same argument as in the proof of \cref{septheo1} implies the claim.

Now, let us assume that \(x_0\) does not exist. Again we find a function \(\hat{g}\in \overbar{\mathfrak{D} \cap C(\hat{X})}^{\inftynorm}\) and \(\lambda \in \mathbb{R}\) satisfying \begin{align*}
    &\hat{g}= 1+\lambda \text{ on } A \text{ and }\; \\
    &\hat{g}=0 +\lambda\text{ on } B.
\end{align*}
If \(\lambda \neq -1\) the claim follows from \cref{septheo1}. Let us assume \(\lambda=-1\).
Since \(x_0\) does not exist, there are, for every point \(x\in A\), a continuous function \(f_x \in \overbar{\mathfrak{D} \cap C(\hat{X})}^{\inftynorm}\) and an open neighbourhood \(U_x\) of \(x\) such that \(f_x\geq 1\) on \(U_x\). Since \(A\) is compact, we can find finitely many \(x_1 \dots x_n \in A\) such that 
\begin{align*}
    \bigcup_{i=1}^n U_{x_i}\supset A.
\end{align*}
This yields, that \(\hat{h}= \left(\sum_{i=1}^n (f_{x_i} \vee 0) \right) \wedge 1 \in \overbar{\mathfrak{D} \cap C(\hat{X})}^{\inftynorm}\) satisfies \(0\leq\hat{h}\leq 1\) and \(\hat{h}=1\) on \(A\). Therefore,
\begin{align*}
   &\hat{g}+\hat{h}=1  \text{ on } A\; \\
    &\hat{g}+\hat{h}\leq 0\text{ on } B.
\end{align*}
Again, an argument like in the proof of \cref{septheo1} implies the claim.
\end{proof}

\section{A nonlinear Riesz-Markov representation theorem on the Dirichlet space}

From now on, we fix a regular, symmetric Dirichlet form \(\mathcal{E}: \lspace \rightarrow \mathbb{R}\cup \{\infty\}\) with core \(\hat{\mathcal{A}}\). We additionally assume that, for every \(u \in \mathfrak{D}\),
\begin{align}
    & u \wedge \cdot: \mathfrak{D}\rightarrow \mathfrak{D},v\mapsto u \wedge v  \text{ is continuous}, \tag{A1} \\
    &-c \vee u \wedge c \rightarrow u \text{ for } c \rightarrow \infty,  \tag{A2}\\
    &-c \vee u \wedge c \rightarrow 0 \text{ for } c \rightarrow 0, \tag{A3}\\
    & \mathfrak{D} \textup{ is separable.} \tag{A4}
\end{align}
Note that the first three assumptions are, for example, satisfied if \(D (\E)=\D\).
Again \(\hat{X}\) denotes the compact space associated with \(\hat{\mathcal{A}}\). By the isomorphism between \(\mathfrak{D}\) and \(\hat{\mathfrak{D}}\), we interpret \(f\in \mathfrak{D}\) as a quasicontinuous function on \(\hat{X}\). In addition \(\hat{X}\) is by definition a compact Hausdorff space, and therefore normal. By capacity we mean the capacity on \(\hat{X}\). All pointwise statements are meant quasi everywhere, if not specified otherwise.
 
 \begin{defn}
  We call a functional \(\phi :\lspace \rightarrow \mathbb{R} \cup \{\infty\}\) \underline{local}, if for every \(u,v \in D(\phi)\),
  \begin{align*}
      \vert u \vert \wedge \vert v \vert =0 \Rightarrow \phi(u+v)=\phi(u)+\phi(v),
  \end{align*}
  and \underline{monotone}, if for every \(u,v \in \lspace\) the inequality \(u \leq v\) implies
  \begin{align*}
      \phi(u)\leq \phi(v).
  \end{align*}
  The set
  \begin{align*}
      \supp(\phi)=\hat{X}\setminus \{& x \in \hat{X} \setdelim \textup{ there is a neighborhood } U \textup{ of }x  \\ &\textup{ such that for all } u \in D(\phi):\supp(u) \subseteq U \Rightarrow \phi(u)=0\}
  \end{align*}
  is called the \underline{support} of \(\phi\).
 \end{defn}

\begin{thm}\label{riez-markov type}
Let \(\mathcal{E}\) be a regular Dirichlet form such that its Dirichlet space \(\D\) satisfies the assumptions (A1) to (A4).
 Let \( \psi: \mathfrak{D}^+ \rightarrow [0,\infty]\) be such that \(\psi (0) =0\). Then the following are equivalent:
 \renewcommand{\labelenumi}{(\roman{enumi})}
\begin{enumerate}
 \item The function \(\psi\) is lower semicontinuous, local, monotone and for every \(u,v \in \mathfrak{D}\)
\[ \psi(u \wedge v)+\psi(u\vee v) \leq \psi(u)+\psi(v).\]
\item There exists a finite, regular Borel measure \(\mu\) on \(\hat{X}\) with \(supp(\mu) \subseteq supp(\psi)\), which is absolutely continuous with respect to the capacity on \(\hat{X}\), that is, every polar set has \(\mu\)-measure \(0\), and a Borel function \(B: \hat{X} \times \mathbb{R}^+ \rightarrow [0,\infty ] \) satisfying
\begin{align}
\label{B1} & B(.,s) \text{ is measurable } & \text{ for every } s \in \mathbb{R} \\
\label{B2} & B(x,0)=0 & \text{ for }\mu \text{-a.e. } x \in \hat{X}\\
\label{B3} & B(x,.) \text{ is lower semicontinuous } & \text{ for }\mu \text{-a.e. } x \in \hat{X} \\
\label{B4} & B(x,.) \text{ is monotone } & \text{ for }\mu \text{-a.e. } x \in \hat{X} 
\end{align}
such that for all \(u \in D(\psi)\)
\begin{align*}
 \psi(u)=\int_{\hat{X}} B(x,u(x)) \; \mathrm{d}\mu.
\end{align*}
\end{enumerate}
\end{thm}

\begin{proof}[Proof of \cref{riez-markov type}  \((ii) \Rightarrow (i)\).]
We first show that \(\psi\) is lower semicontinuous. For this, let \(u\in\mathfrak{D}^+\), and let \((u_n)_n\) be a sequence in \(\mathfrak{D}^+\) converging to \(u\) in \(\mathfrak{D}\). If \(\liminf_{n\to\infty} \psi (u_n) = \infty\), then the inequality \(\psi (u) \leq \liminf_{n\to\infty} \psi (u_n)\) is clearly true. So we assume that \(\liminf_{n\to\infty} \psi (u_n)\) is finite. We can choose a subsequence of \((u_n)_n\), denoted for simplicity again by \((u_n)_n\), such that the limit inferior actually is a limit. By \cref{thm:D_convergnes_implies_pointwise_quasi_everywhere}, there is a subsequence, denoted for simplicity again by \((u_n)_n\), converging quasi everywhere to \(u\), that is, there is a polar set \(A \subseteq \hat{X}\) such that \(u_n\) converges pointwise everywhere on \(\hat{X} \setminus A\). Since \(\mu\) is absolutely continuous with respect to \(\normcap\), \(\mu(A)=0\). The lower semicontinuity of \(B\) and Fatou's lemma yield
\begin{align*}
 \psi(u) &= \int_{\hat{X}} B(x,u(x)) \; \mathrm{d}\mu 
 \\&=\int_{\hat{X} \setminus A} B(x,u(x)) \; \mathrm{d}\mu 
 \\& \leq \int_{\hat{X} \setminus A} \liminf_{n\rightarrow \infty} B(x,u_n(x)) \; \mathrm{d}\mu
\\& \leq \liminf_{n\rightarrow \infty} \int_{\hat{X} \setminus A}  B(x,u_n(x)) \; \mathrm{d}\mu \\& \leq \liminf_{n\rightarrow \infty} \int_{\hat{X}}   B(x,u_n(x)) \; \mathrm{d}\mu \\&= \liminf_{n\rightarrow \infty}  \psi(u_n).
\end{align*}
Thus, \(\psi(u)\) is lower semicontinuous. The functional is local, since the integral is local, and it is monotone, since the integral is monotone and \(B\) is monotone in the second component. In addition,
\begin{align*}
    \psi(u \wedge v)+\psi(u\vee v) &= \int_{\hat{X}} B(x,u(x) \wedge v(x)) \; \mathrm{d}\mu+\int_{\hat{X}} B(x,u(x) \vee v(x) ) \; \mathrm{d}\mu
    \\& = \int_{\{u\leq v\}} B(x,u(x)) \; \mathrm{d}\mu+ \int_{\{u>v\}} B(x,v(x)) \; \mathrm{d}\mu
    \\ &\quad +\int_{\{u \leq v\}} B(x,v(x)) \; \mathrm{d}\mu+\int_{\{u>v\}} B(x,u(x)) \; \mathrm{d}\mu \\
    &=\int_{\hat{X}} B(x,u(x)) \; \mathrm{d}\mu+\int_{\hat{X}} B(x,v(x))\; \mathrm{d}\mu \\
    &=\psi(u)+\psi(v)
\end{align*}
for every \(u,v \in \mathfrak{D}\).
\end{proof}

We prove the converse implication via several lemmas. We define a measure for every bounded \(u\in D(\psi )\) and use the Radon-Nikodym theorem to build the desired integral representation.

\begin{defn}
 Let \(u \in D(\psi)\) be a bounded function and \(K \subseteq \hat{X}\) be a compact subset. We define
\begin{align*}
 R_u(K) = \{ f \in \mathfrak{D} \setdelim f \geq u \text{ on an open neighborhood } U \text{ of } K\}.
\end{align*}
 We define the set function  \(\mu_u: \{ K\subseteq\hat{X} \setdelim K \text{ is compact}\} \rightarrow \mathbb{R}\) by 
\begin{align*}
 \mu_u(K)= \inf_{f \in R_u(K)} \psi(f).
\end{align*}
\end{defn}

We observe, that the infimum is always finite, since \(u \in R_u(K)\).

\begin{lem}\label{measure additive}
 Let \(u \in D(\psi)\) be bounded and \(K_1,K_2 \subseteq \hat{X}\) be disjoint, compact subsets. Then
\begin{align*}
 \mu_u(K_1\cup K_2)=  \mu_u(K_1)+ \mu_u(K_2).
\end{align*}
\end{lem}
\begin{proof}
Since the sets \(K_1,K_2\) are compact and have empty intersection and \(\hat{X}\) is normal, there are open sets \(O_1,O_2\), with disjoint closure, separating \(K_1,K_2\) and open sets \(U_1,U_2\) such that \(K_i \subseteq U_i \subseteq \overbar{U_i} \subseteq O_1\). To construct such a \(U_i\), we take a continuous function \(a \in C(\hat{X})\) that is \(0\) on \(K_i\) and \(1\) on \(O_i^c\) and define
\begin{align*}
    U_i = a^{-1}\left(\left(-\infty, \frac{1}{2} \right)\right).
\end{align*}
This function \(a\) exists by Urysohn's Lemma.

By \cref{septheo2}, there are \(f_1,f_2 \in \mathfrak{D}\) such that
\(f_i \geq u\) on \(U_i\) and \(f_i=0\) on \(O_i^c\). Hence \(f_1\wedge f_2=0\) and \(f_i \in R_u(K_i)\). Now let \(f \in R_u(K_1\cup K_2)\). Note that \(g_i=f\wedge f_i \in R(K_i)\) and
\begin{align*}
g_1+g_2 =g_1\vee   g_2  \leq  f.
\end{align*}
By monotonicity and locality,
\begin{align*}
 \psi( f) &\geq \psi( g_1+ g_2) \\&= \psi(g_1)+\psi(g_2) \\&\geq \mu_u(K_1)+\mu_u(K_2).
\end{align*}
Since \(f\) was arbitrary, \( \mu_u(K_1\cup K_2) \geq  \mu_u(K_1)+ \mu_u(K_2) \). In order to prove the converse inequality, we choose functions \(h_i \in R(K_i)\). Then
\begin{align*}
 \psi(h_1)+\psi(h_2) &\geq \psi( h_1 \wedge h_2 )+\psi( h_1 \vee  h_2)
\\&\geq \psi( h_1 \vee h_2)\\&\geq \mu_u(K_1\cup K_2).
\end{align*}
Again, since the \(h_i\) are arbitrary we conclude 
\begin{align*}
  \mu_u(K_1\cup K_2) \leq  \mu_u(K_1)+ \mu_u(K_2).
\end{align*}
This inequality together with the reversed inequality, which we showed at the beginning of the proof, implies
\begin{align*}
    \mu_u(K_1\cup K_2) =  \mu_u(K_1)+ \mu_u(K_2).
\end{align*}
\end{proof}

\begin{lem}\label{measure monotone}
 Let \(u \in D(\psi)\) be bounded and \(K_1,K_2 \subseteq \hat{X} \) be compact subsets such that \(K_1 \subseteq K_2\). Then
\begin{align*}
 \mu_u(K_1) \leq \mu_u(K_2).
\end{align*}
\end{lem}

\begin{proof}
 This follows immediately from the inclusion \(R_u(K_2) \subseteq R_u(K_1)\). 
\end{proof}

\begin{lem}\label{measure submodular}
  Let \(u \in D(\psi)\) be bounded and \(K_1,K_2 \subseteq \hat{X}\) be compact subsets. Then
\begin{align*}
 \mu_u(K_1\cup K_2) +  \mu_u(K_1\cap K_2) \leq   \mu_u(K_1)+ \mu_u(K_2).
\end{align*}
\end{lem}

\begin{proof}
 Let \(h_i \in R_u(K_i)\). Then
\begin{align*}
 \psi(h_1)+\psi( h_2) &\geq \psi( h_1 \wedge h_2)+\psi ( h_1 \vee  h_2)
\\& \geq \mu_u(K_1\cap K_2)+\mu_u(K_1\cup K_2) .
\end{align*}
Taking the infimum over all \(h_i \in R(K_i)\) implies the claim
\end{proof}

\begin{lem}\label{downward continuous}
 Let \(u \in D(\psi)\) be bounded and let \(\mathcal{M}\) be a downward directed family of compact subsets of \(\hat{X}\), that is, for every \(U,V \in \mathcal{M}\) there is a set \(W \in \mathcal{M}\) with \(W \subseteq U,V\). Then, for \(K= \bigcap_{A \in \mathcal{M}} A\),
\begin{align*}
 \inf_{A \in \mathcal{M}} \mu_u(A)=\mu_u(K).
\end{align*}
\end{lem}

\begin{proof}
 Since, for every \(A\in \mathcal{M}\), \(K \subseteq A\) and, by \cref{measure monotone}, \(\mu_u\) is monotone, \(\inf_{A \in \mathcal{M}} \mu_u(A) \geq \mu_u(K)\). For the converse inequality, let us choose a function \(f \in R_u(K)\) and an associated open neighborhood \(U\) of \(K\) such that \(u \leq f\) on \(U\). Then,
 \begin{align*}
     \emptyset= K\cap U^c = \bigcap_{A \in \mathcal{M}} A \cap U^c .
 \end{align*}
 Note that every \(A \in \mathcal{M}\) and \(U^c\) are closed sets. Since \(\hat{X}\) is a compact Hausdorff space, \(\hat{X}\) has the finite intersection property. That is, if \(\mathcal{C}\) is a collection of closed sets and \(\bigcap \mathcal{C} = \emptyset\), then there are finitely many \(A_1,\dots, A_n\) such that \(\bigcap_{i=1}^n A_i = \emptyset\). Hence, there are finitely many \(A_1,\dots,A_n \in \mathcal{M}\) such that
 \begin{align*}
     \emptyset= K\cap U^c = \bigcap_{A \in \mathcal{M}} A \cap U^c = \bigcap_{i=1}^n A_i \cap U^c.
 \end{align*}
 Since \(\mathcal{M}\) is downward directed and there are only finitely many \(A_i\), there exists a set \(A \in \mathcal{M}\) such that \(A \subseteq \bigcap_{i=1}^n A_i\). Thus,
 \begin{align*}
     \emptyset = \bigcap_{i=1}^n A_i \cap U^c \supseteq A \cap U^c.
 \end{align*}
 Therefore \(A \subseteq U\). Since \(A\) is compact, \(U\) is an open neighborhood of \(A\) and \(f \in R_u(A)\). This yields 
 \begin{align*}
     \inf_{A' \in \mathcal{M}} \mu_u(A') \leq \mu_u(A) \leq \psi(f).
 \end{align*}
 Since \(f\) was arbitrary, 
 \begin{align*}
     \inf_{A' \in \mathcal{M}} \mu_u(A') \leq \mu_u(K),
 \end{align*}
 which implies the claim.
\end{proof}

\begin{lem}
Let \(u \in D(\psi)\) be bounded. The set function \(\mu_u\) can be uniquely extended to a finite regular Borel measure on \(\hat{X}\) such that \(\mu_u(\hat{X})=\psi(u)\) and \(\supp(\mu_u)\subseteq \supp(\psi)\).
\end{lem}

\begin{proof}
 By \cref{measure additive}, \cref{measure monotone} and \cref{measure submodular},  \(\mu_u\) is a monotone, additive function defined on the compact subsets of \(\hat{X}\), which satisfies
 \begin{align*}
      \mu_u(K_1\cup K_2) \leq   \mu_u(K_1)+ \mu_u(K_2),
 \end{align*}
 for every compact sets \(K_1,K_2 \subseteq \hat{X}\). In addition, by \cref{downward continuous}, \(\mu_u\) is downward \(\tau\) continuous \cite[p. 12]{Konig_measure_theroy}. Hence, by \cite[Theorem 9.1 and 9.6]{Konig_measure_theroy}, \(\mu_u\) can be uniquely extended to a regular Radon measure on the Borel \(\sigma\)-algebra. We denote this extension again by \(\mu_u\).
 Since \(\hat{X}\) is compact and \(u \in R_u(\hat{X})\), we have \(\mu_u(\hat{X})\leq \psi(u)\). On the other hand, \(f \in R_u(\hat{X})\) implies \(f\geq u\). The monotonicity of \(\psi\) yields \(\mu_u(\hat{X}) \geq \psi(u)\). Hence, \(\mu_u(\hat{X}) = \psi(u)\).
\end{proof}

\begin{lem}\label{thm:riesz_markov_measure_monoton_in_functionarg}
Let \(u,v\in D(\psi)\) be bounded functions such that \(u\leq v\). Then \(\mu_u\leq \mu_v\).
\end{lem}

\begin{proof}
 Since the measures are regular, it is enough to show \(\mu_u(K)\leq \mu_v(K)\) for every compact \(K\subseteq \hat{X}\). Let \(K\) be such a compact set, \(f_u \in R_u(K)\) and \(f_v \in R_v(K)\). Then \(f_u\wedge f_v \in R_u(K)\), and hence 
 \begin{align*}
     \mu_u(K)\leq \psi(f_u\wedge f_v) \leq \psi( f_v).
 \end{align*}
 This implies \(\mu_u(K) \leq \mu_v(K)\).
\end{proof}

\begin{lem}\label{measure convergence}
Let \((u_n)_n\) be a sequence in \(D(\psi)\) converging to \(u\in D(\psi)\) in \(\mathfrak{D}\) such that \(u_n \leq u\) for every \(n \in \mathbb{N}\). Assume that \(u\) is bounded. Then, for every measurable set \(G\),
\begin{align*}
    \lim_{n\rightarrow \infty}\mu_{u_n}(G)=\mu_u(G).
\end{align*}
\end{lem}

\begin{proof}
Let G be a measurable set. Since \(u_n \leq u\), we have \(\mu_{u_n}(G) \leq \mu_u(G)\) by \cref{thm:riesz_markov_measure_monoton_in_functionarg}. Therefore,
\begin{align*}
    0&\leq \limsup_{n \rightarrow \infty} \big( \mu_u(G) - \mu_{u_n}(G) \big) \\ &\leq
    \limsup_{n \rightarrow \infty} \big( \mu_u(G) - \mu_{u_n}(G)+\mu_u(G^c) - \mu_{u_n}(G^c) \big) \\& \leq \limsup_{n \rightarrow \infty} \big(\mu_u(\hat{X}) - \mu_{u_n}(\hat{X}) \big)\\
    &= \limsup_{n \rightarrow \infty} \big(\psi(u) - \psi(u_n) \big)
    \\&=\psi(u) - \liminf_{n\rightarrow \infty} \psi(u_n) \\&\leq 0 .
\end{align*}
Hence, \(\lim_{n\rightarrow \infty}\mu_{u_n}(G)=\mu_u(G)\).
\end{proof}

\begin{lem}\label{mu u asb continuity}
Let \(u \in D(\psi)\) be bounded, and let \(A \subseteq \hat{X}\) be a polar subset. Then \( \mu_u(A)=0\).
\end{lem}

\begin{proof}
Since the measure \(\mu_u\) is regular we only have to show the claim for compact sets. Let \(K\subseteq \hat{X}\) be a polar, compact set. By definition, there exists a sequence of open sets \((O_n)_n\) such that \( \normcap(O_n) \leq \frac{1}{n}\) and \(K \subseteq O_n\). Without loss of generality we assume that \( O_n \subseteq \overbar{O_n} \subseteq O_{n-1}\). Since if for some arbitrary \(n \in \mathbb{N}\) the inclusion \( O_n \subseteq \overbar{O_n} \subseteq O_{n-1}\) does not hold, there is a continuous function \(f \in C(\hat{X})\) such that \(f=1\) on \(K\) and \(f=0\) on \( O_{n-1}^c\). Then 
\begin{align*}
    O_n \cap f^{-1}\left( \left( - \infty, \frac{1}{2}\right) \right)
\end{align*} 
satisfies our assumptions. Note that such a continuous function exists, since \(\hat{X}\) is normal. 

Moreover, for every \(n \in \mathbb{N}\) we find a function \(w_n \in \mathfrak{D}\) such that \(w_n\geq 1 \) on \(O_n\) and \(\Vert w_n \Vert_\D \rightarrow 0\). 

By \cref{septheo2}, there is a sequence \((z_n)_n\) in \(\mathfrak{D}^+\) such that \(z_n \geq u \) on \(O_{n+1}\) and \(\supp(z_n) \subseteq O_n\). Let \(k \in \mathbb{N}\). We can find \(m_k \in \mathbb{N}\) such that 
\begin{align*}
    \Vert z_k \wedge \Vert u \Vert_\infty w_m \Vert_\D \leq \frac{1}{k}
\end{align*}
for every \(m\geq m_k\). Let us choose an increasing sequence \((k_n)_n\) such that \({k_n} \geq m_{k_{n-1}}\) and define \(U_n= O_{k_n}\). Then \((f_n)_n\), defined by \[f_n= u \wedge z_{k_n}  \wedge \Vert u \Vert_\infty w_{m_{k_n} }, \] is a sequence converging to \(0\) in \(\mathfrak{D}\) and, in addition, satisfies \(f_n= u\) on \(U_{n+1}\) and \(\supp(f_n) \subseteq U_n\).

This implies \((u-f_n) \wedge f_{n+1} = 0\) and \(f_n \in R_u(K)\) for every \(n \in \mathbb{N}\). We compute
\begin{align*}
    \mu_u(\hat{X})&= \psi(u) \\&\geq \psi( u - f_n +  f_{n+1}) \\&=\psi( u-f_n)+\psi( f_{n+1}) \\&\geq 
    \mu_{ u -f_n}(\hat{X}) + \mu_u(K).
\end{align*}
We observe that \(u - f_n \rightarrow u\) and  \(u-f_n \leq u\). Hence, the preceding lemma yields
\begin{align*}
\mu_u(\hat{X})\geq \lim_{n\rightarrow \infty}  \mu_{ u -f_n}(\hat{X}) + \mu_u(K) =\mu_u(\hat{X})+ \mu_u(K),
\end{align*}
or, in other words, \(\mu_u(K)=0\).
\end{proof}
Let \(f_1,f_2\) be two representatives of an element \(f \in \mathfrak{D}\). By \cref{thm:unique_qc_representatives}, \(f_1=f_2\) quasi everywhere. Hence, there is a polar set \(A\) such that \(\{f_1=0\} \cup A=\{f_2=0\} \cup A\). Thus, for an arbitrary bounded \(u \in D(\psi)\)
\begin{align*}
    \mu_u(\{f_1=0\}) &\leq \mu_u(\{f_1=0\} \cup A) \\&= \mu_u(\{f_2=0\} \cup A) \\& \leq \mu_u(\{f_2=0\}) + \mu_u(A)\\&=\mu_u(\{f_2=0\}),
\end{align*}
since \(\mu_u(A)=0\) by \cref{mu u asb continuity}. The converse inequality \(\mu_u(\{f_1=0\}) \geq \mu_u(\{f_2=0\})\) follows by symmetry. Therefore \(\mu_u(\{f_1=0\}) = \mu_u(\{f_2=0\})\) and \(\mu_u(\{f=0\})\) is well defined, since it is independent of the choice of the representative.

\begin{lem}\label{measure does not charge zero}
Let \(u \in D(\psi) \) be bounded. Then \(\mu_u(\{u=0\})=0\).
\end{lem}

\begin{proof}
Since \(u\) is quasicontinuous, the set \(U_k=\{u < \frac{1}{k}\}\) is quasi open. Hence, for a fixed \(k\), there are open sets \(O^k_n\) and functions \(f^k_n \in \mathfrak{D}\) such that \(0 \leq f^k_n\leq 1\), \(f^k_n=1\) on \(O^k_n\), \( \lim_{n\rightarrow \infty} \Vert f^k_n \Vert_\D=0\) and \(U_k \cup O^k_n\) is open. Let us fix a \(k \in \mathbb{N}\) and define 
\begin{align*}
    g_n= (u-\frac{1}{k})^+ \wedge \Vert(u-\frac{1}{k})^+\Vert_\infty f^k_{n}.
\end{align*}
Since \(f^k_n \rightarrow 0\) for \(n \rightarrow \infty\) and the pointwise infimum is continuous on \(\D\), we can find an \(m_k \in \mathbb{N}\) large enough such that
\begin{align*}
   \Vert g_{m_k} \Vert = \big\Vert (u-\frac{1}{k})^+ \wedge \Vert(u-\frac{1}{k})^+\Vert_\infty f^k_{m_k} \big\Vert_\D \leq \frac{1}{k}.
\end{align*}
Let us define
\begin{align*}
   u_{k} =  (u-\frac{1}{k})^+ - g_{m_k},
\end{align*}
which implies 
\begin{align*}
\Vert (u-\frac{1}{k})^+ -u_{k} \Vert_\D =   \Vert (u-\frac{1}{k})^+ -(u-\frac{1}{k})^+ +g_{m_k} \Vert_\D  =\Vert g_{m_k} \Vert_\D \leq \frac{1}{k}.
\end{align*}
This construction works for every \(k \in \mathbb{N}\). Thus, we have a sequence \((u_k)_k\) satisfying 
\begin{align*}
    \Vert u-u_k\Vert_\D &=\Vert u-(u-\frac{1}{k})^+ +(u-\frac{1}{k})^+ +u_k \Vert_\D 
    \\ & \leq \Vert u-(u-\frac{1}{k})^+\Vert_\D + \Vert(u-\frac{1}{k})^+ +u_k \Vert_\D
    \\ & \leq \Vert u-(u-\frac{1}{k})^+\Vert_\D +\frac{1}{k} \rightarrow 0,
\end{align*}
since \((u-\frac{1}{k})^+=u- u \wedge \frac{1}{k} \rightarrow u\) by Assumption (A3).
Therefore, \(u_k\rightarrow u\) and \(u_k\leq u \). In addition, \(u_k=0\) on \(U_k\), since \( (u-\frac{1}{k})^+ =0 \) on \( U_k\) and therefore \(g_{m_k}=0\), which in turn implies \(u_k=0\) on \(U_k\). On the other hand, \(u_k=0\) on \(O^k_{m_k}\), since \(f^k_{m_k}=1\) on \(O^k_{m_k}\) which yields \(g_{m_k}= (u-\frac{1}{k})^+\) on \(O^k_{m_k}\) and thus, \(u_k=0\) on \(O^k_{m_k}\).
Hence \(u_k=0\) on the open set \(U_k \cup O^k_{m_k}\) which is an open neighbourhood of \(\{u=0\}\).

Let \(K\subseteq \{u=0\}\) be a compact set and \(k\) be fixed. By the properties of \(u_k\), \(K \cap \supp(u_k) = \emptyset\). Therefore we can find a function \(h \in \mathfrak{D}\) such that \(h \geq u\) on some neighbourhood \(U\) of \(K\) and \(h=0 \text{ on } \supp(u_k)\). Hence, \(h \in R_u(K)\), in particular \(u_k \wedge h \in R_{u_k}(K)\), and \(u_k \wedge  h =0\). Since \(\psi(0)=0\), then \(\mu_{u_k}(K)=0\) and 
\[ 
0=\lim_{k\rightarrow \infty}\mu_{u_k}(K)=\mu_u(K).
\] 
Inner regularity implies the claim.
\end{proof}

\begin{lem}
Let \(u,v \in D(\psi)\) both be bounded. Then \(\mu_u(G)\leq \mu_v(G)\) for every Borel set \(G\subseteq \{u\leq v\}\). \label{measure inequ}
\end{lem}

\begin{proof}
At first, let \(K \subseteq U=\{u<v\}\) be compact. Since \( U\) is quasi open, there are open sets \(O_n\) and functions \(w_n \in \mathfrak{D}\), satisfying that \(U_n=U \cup O_n\) is open, \(0\leq w_n \leq 1\), \(w_n=1\) on \(O_n\) and \( \Vert w_n \Vert_\D \rightarrow 0\).
We define \(c=\max( \Vert u \Vert_\infty,\Vert v \Vert_\infty),u_n= u \wedge C(1-w_n),v_n= v \wedge C(1-w_n)\). Then \(u_n=v_n=0\) on \(O_n\) and \(u_n<v_n\) on \(U\). Hence, \(u_n\leq v_n\) on the open set \(U_n\). Since \(K\subseteq U \subseteq U_n\), we can find functions \(f_n\in \mathfrak{D}\), such that \(0 \leq f_n\leq v_n\), \(f_n=v_n\) on some open neighbourhood \(A_n \subseteq U_n\) of \(K\) and \(\supp(f_n) \subseteq U_n\).

Now, let \( h \in R_{v_n}(K)\) be arbitrary. We observe \(h \wedge f \in R_{u_n}(K)\) and
\begin{align*}
    \psi( h ) &\geq  \psi(h\wedge f) \\
    & \geq \mu_{u_n}(K).
\end{align*}
Since \(h\) was arbitrary, the preceding lemma implies 
\begin{align*}
    \mu_{v_n}(K)\leq \mu_{u_n}(K),
\end{align*}
and taking the limit yields, by inner regularity,
\begin{align*}
    \mu_u(G)\leq \mu_v(G)
\end{align*}
for every Borel set \(G\subseteq \{u< v\}\).

For the second part, let \(G \subseteq \{u\leq v\} \) and \((\lambda_n)_n\) be a sequence in \(\mathbb{R}\) increasing to \(1\). Then, by \cref{measure does not charge zero},
\begin{align*}
    \mu_u(G\cap \{v=0\})=\mu_v(G\cap \{v=0\})=0,
\end{align*}
 and
\begin{align*}
    \mu_{\lambda_n u}(G \cap \{0<v\}) \leq \mu_{v}(G \cap \{0<v\}), 
\end{align*}
since \(G \cap \{0<v\}=\{ u\leq v \text{ and } v > 0 \} \subseteq \{ \lambda_n u < v\}\) . After taking the limit \( n \rightarrow \infty\), we obtain
\begin{align*}
    \mu_u(G)&= \mu_u(G\cap \{v=0\}) + \mu_{u}(G \cap \{0<v\})
    \\& \leq \mu_v(G\cap \{v=0\}) + \mu_{v}(G \cap \{0<v\})
    \\& \leq \mu_v(G),
\end{align*}
which is the claim.
\end{proof}

From now on let \(\{f_n\setdelim n\in \mathbb{N}\}\) be a subset of \(D(\psi)\) such that \(\overbar{\{f_n\setdelim n\in \mathbb{N}\}}^{\D} \supseteq D(\psi) \). Note that such a subset exists, since we assumed that \(\D\) is separable. Without loss of generality, we may assume that the \(f_n\) are bounded. If they are not, then we may consider instead the set \( \{0 \vee f_n \wedge k \setdelim n,k \in \mathbb{N}\}\), which is still countable and dense.
We define the finite regular measure
\begin{align*}
  \mu(G)= \sum_{n=1}^\infty \frac{1}{2^n} \frac{\mu_{f_n}(G)}{1+\mu_{f_n}(\hat{X})}. 
\end{align*}
Let \(A \subseteq \hat{X}\) be polar. \cref{mu u asb continuity} shows that \(\mu_{f_n}(A)=0\) for every \(n \in \mathbb{N}\). Hence, \(\mu(A)=0\), or in other words,  \(\mu\) is absolutely continuous with respect to the norm capacity.

\begin{lem} \label{lem.abscont}
Let \(u \in D(\psi)\) be bounded. Then \(\mu_u\) is absolutely continuous with respect to \(\mu\).
\end{lem}

\begin{proof}
Let \(G\) be a Borel set such that \(\mu(G)=0\). Then \(\mu_{f_n}(G)=0\) for every \(n \in \mathbb{N}\). Since \(\{f_n\setdelim n\in \mathbb{N}\}\) is dense in \(D(\psi)\), there is a subsequence \((f_{n_k})_k\) satisfying \(\lim_{k\rightarrow \infty} f_{n_k} =u \) in \(\D\). Then \(u_k=f_{n_k} \wedge u\) converges to \(u\) in \(\D\), since we assumed the lattice operations to be continuous.  The inequality \cref{measure inequ} yields
\begin{align*}
    0 \leq \mu_{u_k}(G) \leq \mu_{f_{n_k}}(G)=0.
\end{align*}
By \cref{measure convergence}, taking the limit yields the claim.
\end{proof}

By the Radon-Nikodym theorem and by the preceding \cref{lem.abscont}, for every bounded \(u \in D(\psi)\) there exists a positive, measurable function \(B_u\) such that
\begin{align*}
    \mu_u(G)=\int_G B_u(x) \; \mathrm{d}\mu .
\end{align*}

The following lemma is a direct consequence of \cref{measure convergence}, \cref{measure does not charge zero} and \cref{measure inequ}.

\begin{lem}
Let \(u_n,u,v \in D_b(\psi)\). Then:
\renewcommand{\labelenumi}{(\roman{enumi})}
\begin{enumerate}
    \item If \(u_n\leq u, u_n \rightarrow_{\D} u\), then \( B_{u_n}\rightarrow B_u\; \mu-\)almost everywhere.
    \item \(B_u=0\) \(\mu-\)almost everywhere on \(\{u=0\}\).
    \item \(B_u\leq B_v\) \(\mu-\)almost everywhere on \(\{u\leq v\}\).
\end{enumerate}
\end{lem}

For every \(x \in \hat{X}\), we define \( I(x)=\textup{convexhull}\{f_n(x),n\in \mathbb{N}\} \) and the function
\begin{align*}
   B(x,s)= 
   \begin{cases}
   \sup_{n\in \mathbb{N}} B_{f_n}(x) 1_{\{f_n<s\}}(x) & \text{if } s \in I(x) \\
    + \infty & \text{else}
   \end{cases}
\end{align*}
for every \(s \in \mathbb{R}^+\).

We are now finally ready to prove the converse direction in our theorem.

\begin{proof}[Proof of \cref{riez-markov type} \((i) \Rightarrow (ii)\)]
At first, we notice that, since the functions \(f_n\), which we have chosen above, are positive for every \(n \in \mathbb{N}\), the sets \(\{f_n<0\}\) are empty for every \(n \in \mathbb{N}\). Hence, \(1_{\{f_n<0\}} =0\), which implies \(B(x,0)=0\) for every \(x \in \hat{X}\). Additionally, the sets \(\{f_n<s\}\) are increasing with \(s\) for every \(n\in \mathbb{N}\). Thus \(B(x,\cdot)\) is monotone.

We prove the lower semicontinuity of \(B(x, \cdot)\). Let \(x \in \hat{X},s\in \mathbb{R}^+\). By the definition of \(B\), for every \(\epsilon >0\) there exists an \(m \in \mathbb{N}\) such that \(f_m(x)<s\) and
\begin{align*}
    B(x,s)-\epsilon\leq B_{f_m}(x) \leq B(x,s).
\end{align*}
This implies \(B(x,s)-\epsilon\leq B(x,t) \leq B(x,s)\) for every \(f_m(x) < t \leq s\). Obviously, \(B(x,s)-\epsilon\leq B(x,t) \) if \(t>s\). Therefore we found a neighborhood \(U\) of \(s\) such that \(B(x,s)-\epsilon\leq B(x,t) \) for every \(t \in U\). As a consequence, \(s \mapsto B(x,s)\) is lower semicontinuous for every \(x \in \hat{X}\).

To show the last part of the claim we consider the case of a bounded function \(u \in D_b(\psi)\) first. For each \(n\in \mathbb{N}\), we have \(B_{f_n}\leq B_u\) for \(\mu\)-almost every \(x \in \{f_n\leq u\}\). Hence, 
\begin{align}\label{eqn:ralph_stern}
    B( \cdot, u(\cdot))\leq B_u(\cdot) \quad \mu \text{-almost everywhere.}
\end{align}

For the converse inequality, we first note that \(B_u(x)=0\) for \(\mu\)-almost all \(x\in \{u=0\}\), and thus \(B_u(x)=0=B(x,0)=B(x,u(x))\) for \(\mu\)-almost all \(x\in \{u=0\}\).
On the other hand, let \((\lambda_m)_{m \in \mathbb{N}}\) be a sequence in \(\mathbb{R}\) strictly increasing to \(1\). Then there exists a subsequence \(f_{n^m_k}\) converging to \(\lambda_m u\) in \(\mathfrak{D}\) and additionally quasi everywhere for every \(m \in \mathbb{N}\). Since \(\mu\) does not charge polar sets, the sequence converges even \(\mu\)-almost everywhere and 
\begin{align*}
    u(x) > \lim_{k \rightarrow \infty} f_{n^m_k} (x)=\lambda_m u(x) > \lambda_{m-1} u(x)
\end{align*}
for \(\mu\)-almost every \(x\in \{u>0\}\). Hence,
\begin{align*}
    B(x,u(x))&= \sup_{f_n(x)<u(x)}B_{f_n}(x) \\
    & \geq \limsup_{k \rightarrow \infty} B_{f_{n^m_k}}(x) 
    \\ &\geq B_{\lambda_{m-1}u}(x)
\end{align*}
for \(\mu\)-almost every \(x\in \{u>0\}\). We already showed \( B_{\lambda_{m}u}(x) \rightarrow B_u(x)\) \(\mu\)-almost everywhere. Thus, we have shown 
\begin{align*}
    B(x,u(x)) \geq B_u(x).
\end{align*}
Therefore, together with \eqref{eqn:ralph_stern} this implies \( B(x,u(x)) = B_u(x)\) \(\mu\)-almost everywhere. Hence,
\begin{align*}
    \psi(u)=\int B_u(x) \; \mathrm{d}\mu= \int B(x,u(x)) \; \mathrm{d}\mu.
\end{align*}
For an arbitrary \(u \in D(\psi)\), we observe that \(\lim_{n\rightarrow \infty} \psi(u \wedge n) =\psi(u)\) by the lower semicontinuity and monotonicity of \(\psi\) and, since \(B(\cdot,\cdot)\) is monotone and lower semicontinuous in the second argument, \(B(x,u(x))=\lim_{n \rightarrow\infty} B(x,u(x)\wedge n) \) \(\mu\)-almost everywhere.
Therefore, the monotone convergence theorem yields
\begin{align*}
    \psi(u)&=\lim_{n\rightarrow\infty}\psi(u \wedge n) \\
    & = \lim_{n\rightarrow\infty} \int_{\hat{X}} B(x,u(x)\wedge n) \; \mathrm{d}\mu \\ 
    & = \int_{\hat{X}} \lim_{n\rightarrow\infty} B(x,u(x)\wedge n) \; \mathrm{d}\mu \\ 
    & = \int_{\hat{X}} B(x,u(x)) \; \mathrm{d}\mu,
\end{align*}
which is our claim.
\end{proof}

\begin{remark}\label{rieszremark}
Let  \(\psi_1,\psi_2\) be two functionals satisfying all assumptions of the theorem. There are measures \(\mu_1\) and \(\mu_2\) given by the previous theorem for \(\psi_1\) and \(\psi_2\), respectively.
The proof also shows one can find a common measure \(\nu\), for example \(\nu = \mu_1+\mu_2\) if \(\mu_1\) and \(\mu_2\) are the representing measures for \(\psi_1\) and \(\psi_2\), respectively, and functions \(B_{1,\nu},B_{2,\nu}\), as in the theorem, such that
\begin{align*}
    \psi_i(u)=\int_{\hat{X}} B_{i,\nu}(x,u(x)) d\nu,
\end{align*}
or in other words, for finitely many functionals one can always choose a common representing measure.
\end{remark}

\section{Domination of local sub-Markovian semigroups}

\begin{defn}
We call a function \(f: \mathbb{R} \rightarrow [0,\infty]\) \underline{bi-monotone} if 
\(f\) is monotone decreasing on \((-\infty,0)\) and monotone increasing on \( (0,\infty)\).
\end{defn}

\begin{thm}\label{main theo}
Let \(\mathcal{E}: \lspace \rightarrow [0,\infty]\) be a quasilinear, symmetric, local, regular Dirichlet form with core \(\hat{\mathcal{A}}\), fulfilling the assumptions (A1) to (A4) and, for every \(u,v \in D(\mathcal{E})\),
\begin{align*}
    \mathcal{E}(u \wedge v) + \mathcal{E}(u \vee v)=\mathcal{E}(u)+\mathcal{E}(v).
\end{align*}
Let \(T\) be the Markovian semigroup generated by \(\mathcal{E}\) and \(S\) a semigroup generated by a convex, lower semicontinuous functional \( \F : \lspace \rightarrow [0,\infty]\). Then \(\F\) is local, \(S\) is order preserving and \(S\) is dominated by \(T\) in the sense that
\begin{align*}
    \vert S_t(u)\vert  \leq  T_t(\vert u \vert ) \text{ for every } t \geq 0 \text{ and every } u \in \lspace,
\end{align*}
if and only if there exists a finite regular Borel measure \(\mu\) on \(\hat{X}\), which is absolutely continuous with respect to  \(\normcap\) on \(\hat{X}\), and a Borel function \(B: \hat{X} \times \mathbb{R} \rightarrow [0,\infty ] \) satisfying
\begin{align*}
 & B(.,s) \text{ is measurable } & \text{ for every } s \in \mathbb{R}, \\
& B(x,0)=0 & \text{ for }\mu \text{-a.e. } x \in \hat{X},\\
& B(x,.) \text{ is lower semicontinuous} & \text{ for }\mu \text{-a.e. } x \in \hat{X}, \\
& B(x,.) \text{ is bi-monotone } & \text{ for }\mu \text{-a.e. } x \in \hat{X}, \\
\end{align*}
such that for all \(u \in D(\F)\)
\begin{align*}
 \F (u)=\mathcal{E}(u)+\int_{\hat{X}} B(x,u(x)) \; \mathrm{d}\mu.
\end{align*}
\end{thm}

\begin{remark}
Since \(\mathcal{E}\) is quasilinear, \cref{thm:continuous_lattice_operations} implies (A1), (A2) and (A3).
\end{remark}

\begin{proof} 
First, let us assume that \(\F\) is local, \(S\) order preserving and \(S\) is dominated by \(T\). By \cref{thm_characterisation_domination_and_abseqn}, domination of \(S\) by \(T\) implies that
 \begin{align*}
 \F( ( \vert u \vert \wedge v ) \textup{sgn}(u)) + \mathcal{E}( \vert u \vert \vee v) \leq \F(u)+  \mathcal{E}(v) \quad (u,\, v\in\lspace ).
 \end{align*}
Let \(f \in D(\F)\). Choosing \(v=0\) and \(u=f^+\) in the previous inequality yields 
\begin{align*}
    \E(f^+) \leq \F(f^+).
\end{align*}
Choosing \(v=0\) and \(u=f^-\) yields
\begin{align*}
    \E(f^-) \leq \F (f^-).
\end{align*}
Hence, \(f^+,f^- \in \D\) and therefore \(f\in \D\), the Dirichlet space associated with \(\E\). Thus
\begin{align*}
    D(\F )\subseteq \mathfrak{D}.
\end{align*}
Hence, the functional \(\psi: \mathfrak{D} \rightarrow [0,\infty]\) given by
\begin{align*}
    \psi(u)=\begin{cases} \F(u)-\mathcal{E}(u) &\text{ if }u \in D(\F), \\ \infty &\text{ otherwise, }
    \end{cases}
\end{align*}
is well defined. As a difference of local functionals, it is local. Let \(u,v\in D(\F)\). Then
\begin{align*}
    \psi(u \wedge v) + \psi(u \vee v)&= \F(u \wedge v)-\mathcal{E}(u \wedge v)+\F(u \vee v)-\mathcal{E}(u \vee v)
    \\&= \F(u \wedge v) +\F(u \vee v)-\mathcal{E}(u) -\mathcal{E}(v) 
    \\& \leq \F(u) +\F(v)-\mathcal{E}(u) -\mathcal{E}(v)
    \\&=\psi(u)+\psi(v).
\end{align*}
For either \(u\) or \(v\) not in \(D(\F)\) the previous inequality is trivial. Hence,
\begin{align*}
    \psi(u \wedge v) + \psi(u \vee v)&\leq \psi(u)+\psi(v)
\end{align*}
for every \(u,v \in \mathfrak{D}\).
Since \(\mathcal{E}\) is quasilinear and therefore continuous on \(\mathfrak{D}\), \(\psi\) is lower semicontinuous.  By \cref{thm_characterisation_domination_and_abseqn}, the domination of the semigroups implies 
 \begin{align*}
 \F( ( \vert u \vert \wedge v ) \textup{sgn}(u)) + \mathcal{E}( \vert u \vert \vee v) \leq \F(u)+  \mathcal{E}(v)
 \end{align*}
 for every \(u,v \in \lspace\) such that \(v \geq 0\). Hence
\begin{align*}
    \F(v)-\mathcal{E}(v)\leq \F(u)-\mathcal{E}(u)
\end{align*}
for all \(u,v \in \mathfrak{D}\), such that \(0\leq v \leq u\). On the other hand, let \(f,g \in \lspace\) such that \(f \leq g \leq 0\). Setting \(u =f\) and \(v=-g\) in the inequality from \cref{thm_characterisation_domination_and_abseqn} leads to
 \begin{align*}
 \F ( - (-g) ) + \mathcal{E}( -f ) \leq \F (f)+  \mathcal{E}(-g).
 \end{align*}
Since \(\E\) is symmetric, that is  \(\mathcal{E}(h)=\mathcal{E}(-h)\) for every \(h \in \lspace\),
\begin{align*}
    \F(g)-\mathcal{E}(g)\leq \F(f)-\mathcal{E}(f).
\end{align*}
Hence \(\psi\) is monotone increasing on the cone of positive functions and monotone decreasing on the cone of negative functions.

Therefore, the functions \( \psi_1,\psi_2: \mathfrak{D}^+ \rightarrow [0,\infty]\), defined by
\begin{align*}
    \psi_1(u)=\psi(u) \text{ and }
    \psi_2(u)=\psi(-u) \quad u \in \mathfrak{D},
\end{align*}
satisfy all the assumptions of \cref{riez-markov type}. Hence, by \cref{riez-markov type} and \cref{rieszremark}, there is a measure \(\mu\) on \(\hat{X}\) and functions \(B_1,B_2: \hat{X} \times \mathbb{R}^+ \rightarrow [0,\infty]\) satisfying the conditions \eqref{B1}-\eqref{B4} from \cref{riez-markov type}, such that the integral representation
\begin{align*}
\psi_i(u)=\int_{\hat{X}}B_i(x,u(x)) \; \mathrm{d}\mu, \; i=1,2    
\end{align*} 
holds for all positive \(u\in\mathfrak{D}^+\). We define \(B:\hat{X} \times \mathbb{R} \rightarrow [0,\infty]\) by
\begin{align*}
    B(x,s)=\begin{cases} B_1(x,s) &\text{ if } s \geq 0 \\ B_2(x,-s) &\text{ otherwise. }
    \end{cases}
\end{align*}

Then \(B\) is measurable in the first argument and bi-monotone and lower semicontinuous in the second argument and, in addition, \(B(x,0)=0\) for \(\mu\)-almost all \(x \in \hat{X}\). For a fixed \(u \in D(\F)\), this yields
\begin{align*}
    \psi(u)&=\psi(u^+ - u^-)=\psi(u^+)+\psi(-u^-)=\psi_1(u^+)+\psi_2(u^-)\\
    &=\int_{\hat{X}} B_1(x,u^+(x)) \; \mathrm{d}\mu +\int_{\hat{X}} B_2(x,u^-(x)) \; \mathrm{d}\mu \\
    &=\int_{\hat{X}} B(x,u(x)) \; \mathrm{d}\mu.
\end{align*}
Therefore,
\begin{align*}
    \F(u)= \mathcal{E}(u)+\int_{\hat{X}} B(x,u(x)) \; \mathrm{d}\mu,
\end{align*}
which is the claim.

Now let us assume we have a measure \(\mu\) and a function \(B\) as in the theorem. It is easy to check, that
\begin{align*}
    \F(f)=\mathcal{E}(f)+\int_{\hat{X}} B(x,f(x))\; \mathrm{d}\mu 
\end{align*}
is local and 
\begin{align*}
    \F(u \wedge v)+\F(u \vee v) \leq \F(u)+ \F(v)
\end{align*}
for every \(u,v\) in \(\mathfrak{D}\). Hence, \(\mathcal{F}\) is local and \(S\) is order preserving. Now, by \cref{thm_characterisation_domination_and_abseqn}, we have
\begin{align*}
    \mathcal{E}( ( \vert u \vert \wedge v ) \textup{sgn}(u)) + \mathcal{E}( \vert u \vert \vee v) \leq \mathcal{E}(u)+  \mathcal{E}(v)
\end{align*}
and, by the properties of the integral and bi-monotonicity,
\begin{align*}
    \int_{\hat{X}} B(x,( \vert u \vert \wedge v ) \textup{sgn}(u)) \; \mathrm{d}\mu &=
    \int_{\{\vert u \vert \leq v\}} B(x,u) \; \mathrm{d}\mu +
    \int_{\{\vert u \vert > v\}} B(x,v \textup{sgn}(u)) \; \mathrm{d}\mu \\
    & \leq \int_{\{\vert u \vert \leq v\}} B(x,u) \; \mathrm{d}\mu +
    \int_{\{\vert u \vert > v\}} B(x,u) \; \mathrm{d}\mu\\
    &\leq \int_{\hat{X}} B(x,u) \; \mathrm{d}\mu.
\end{align*}
Both inequalities together imply
\begin{align*}
     \F( ( \vert u \vert \wedge v ) \textup{sgn}(u)) + \mathcal{E}( \vert u \vert \vee v) \leq \F(u)+  \mathcal{E}(v).
\end{align*}
Hence, by \cref{thm_characterisation_domination_and_abseqn}, \(S\) is dominated by \(T\).
\end{proof}

\begin{corollary}\label{main theo coro}
In the context of the previous theorem, let there be a third local and order preserving semigroup \(R\) generated by a convex, lower semicontinuous functional \(\mathcal{G}\) such that
\begin{align*}
    \vert R_t(u)\vert  \leq  S_t(\vert u \vert )
\end{align*}
and
\begin{align*}
    \vert R_t(u)\vert  \leq  -S_t(-\vert u \vert ).
\end{align*}
Then there exists a function \(B_{\mathcal{G}}\) satisfying the same properties as \(B\) (the conditions \eqref{B1}-\eqref{B4}) such that 
\begin{align*}
    \mathcal{G}(u)=\mathcal{E}(u)+\int_{\hat{X}} B_{\mathcal{G}} (x,u(x)) \; \mathrm{d}\mu
\end{align*}
for \(u \in D(\mathcal{G})\). If there is closed subset \(A\subseteq \hat{X}\) such that
\begin{align*}
    \mathcal{G}(u)=\mathcal{E}(u)+\int_{A} B_{\mathcal{G}} (x,u(x)) \; \mathrm{d}\mu
\end{align*}
for \(u \in D(\mathcal{G})\) then
\begin{align*}
    \F(u)=\mathcal{E}(u)+\int_{A} B_{\F} (x,u(x)) \; \mathrm{d}\mu.
\end{align*}
\end{corollary}

\begin{proof}
\cref{main theo} and \cref{rieszremark} imply the existence of \(B_{\mathcal{G}}\) such that 
\begin{align*}
    \mathcal{G}(u)=\mathcal{E}(u)+\int_{\hat{X}} B_{\mathcal{G}} (x,u(x)) \; \mathrm{d}\mu
\end{align*}
holds. By \cref{thm_characterisation_domination_and_abseqn}, the two sided domination we assumed yields
\begin{align*}
  \int_{\hat{X}} B_{\mathcal{G}} (x,u(x)) \; \mathrm{d}\mu \geq \int_{\hat{X}} B_{\F} (x,u(x)) \; \mathrm{d}\mu.  
\end{align*}
Note that we now need both domination assumptions, in contrast to the previous theorem, since \(\F\) is not necessarily symmetric and \(-S_t(- \cdot)\) is the semigroup generated by \( \F(- \cdot)\), see \cite{ChWa19} for a similar discussion.

Now let us assume there is a set \(A\) as in the statement and \(U=A^c\). Thus, \(\mathcal{E}(f)=\mathcal{G}(f)=\F(f)\) for every \(f \in \mathfrak{D}_A\). Let us assume, that there is a function \(u \in D(\mathcal{G})\) such that 
\begin{align*}
    \F(u)>\mathcal{E}(u)+\int_{A} B_{\F} (x,u(x)) \; \mathrm{d}\mu,
\end{align*}
that is 
\begin{align*}
   \int_{U} B_{\F}(x,u(x)) \; \mathrm{d}\mu>0.
\end{align*}
Since \(\mu\) is regular, the net \((1_K)_{K\subseteq U \atop K \text{ compact}}\) converges pointwise \(\mu\)-almost everywhere to \(1_U\). Hence, by the dominated convergence theorem, there is a compact set \(K \subseteq U\) such that
\begin{align*}
    \int_{K} B_{\F} (x,u(x)) \; \mathrm{d}\mu>0.
\end{align*}
By the separation theorems, there is a function \(v\) such that \(v=0\) on A and \(v=u\) on \(K\). Therefore 
\begin{align*}
    \int_{K} B_{\F}(x,v(x)) \; \mathrm{d}\mu=0,
\end{align*}
since \(v \in \mathfrak{D}_A\). But 
\begin{align*}
    0=\int_{K} B_{\F}(x,v(x)) \; \mathrm{d}\mu=\int_{K} B_{\F}(x,u(x)) \; \mathrm{d}\mu.
\end{align*}
This is a contradiction. Thus, no such \(u\) exists and 
\begin{align*}
    \F(f)=\mathcal{E}(f)+\int_{A} B_{\F}(x,f(x)) \; \mathrm{d}\mu
\end{align*}
for every \(f \in \mathfrak{D}\).
\end{proof}

\section{Applications}

Let us illustrate \cref{main theo} and \cref{main theo coro} with the help of some examples. At first we consider the \(p(x)\)-Laplacian.
\begin{exmp}
Let \(X\) be a bounded \(C^1\) domain in \(\mathbb{R}^n\) and \(\E^{p(x)}:\lspace \rightarrow [0,\infty]\) the energy of the \(p(x)\)-Laplacian with Neumann boundary conditions given by
\begin{align*}
    \E^{p(x)}_N (f)= \int_X \frac{ \vert \nabla f(x) \vert ^{p(x)}}{p(x)} \; \mathrm{d}x
\end{align*}
for \(f \in \lspace\). We assume \(p: X \rightarrow \mathbb{R}\) is a measurable function, such that \(1 < \inf p \leq \sup p < \infty \) and, in addition, that \(1/p\) is log-Hölder continuous. Then \cite[Theorem 9.1.7]{DHHR_Modular_Spaces} states that \(C^\infty(\overbar{X})\) is densely embedded into \(\mathfrak{D}\). Hence, \(\hat{\mathcal{A}}=C(\overbar{X})\) is a regular core of \(\E^{p(x)}_N\). Therefore, \(\hat{X}=\overbar{X}\). The assumptions (A1)-(A3) are fulfilled, since \(\E^{p(x)}_N\) is quasilinear and \cite[Theorem 8.1.6]{DHHR_Modular_Spaces} implies that \(\mathfrak{D}\) is separable, which is the assumption (A4). Let \(T_N\) be the semigroup generated by \(\E^{p(x)}_N\) and \(S\) another order preserving semigroup generated by a convex, local, lower semicontinuous functional \(\F\) such that for every \(u \in \lspace\),
\begin{align*}
    \vert S(t) (u)\vert  \leq  T_N(t) (\vert u \vert ).
\end{align*}
Then \cref{main theo} implies, that there is a positive measure \(\mu\) on \(\overbar{X}\) and a function \(B: \overbar{X}\times \mathbb{R} \rightarrow [0,\infty]\) satisfying the properties stated in \cref{main theo} such that
\begin{align*}
    \F (u)=\int_X \frac{ \vert \nabla u(x) \vert^{p(x)}}{p(x)} \; \mathrm{d}x + \int_{\overbar{X}} B(x,u(x)) \; \mathrm{d}\mu.
\end{align*}
For example the \(p(x)\)-Laplacian with Dirichlet boundary conditions, given by 
\begin{align*}
    \E_D^{p(x)}(u)= \begin{cases} \int_X \frac{|\nabla u(x)|^{p(x)}}{p(x)} \; \mathrm{d}x &\text{ if } u\in \mathfrak{D}_0=\overbar{C^\infty_C(X)}, \\
                     \infty &\text{ otherwise }
                     \end{cases} 
\end{align*}
generates such a semigroup. In fact we can write 
\begin{align*}
\E_D^{p(x)}(u)=\int_X \frac{|\nabla u(x)|^{p(x)}}{p(x)} \; \mathrm{d}x+\int_{\overbar{X}} B(x,u(x)) \; \mathrm{d}\mu,
\end{align*}
where \(\mu\) is the \((n-1)\)-dimensional Hausdorff measure and 
\begin{align*}
    B(x,s)=\begin{cases}
    0 & \text{ if } x \in X \\
    0 & \text{ if } x \in \partial X, s=0 \\
    \infty & \text{ else. }
    \end{cases}
\end{align*}
Let us denote by \(T_D\) the semigroup generated by the  \(p(x)\)-Laplacian with Dirichlet boundary conditions.
\cref{main theo coro} implies, that if a local and order preserving semigroup \(S\) generated by a convex and lower semicontinuous functional \(\F\) is sandwiched between \(T_N\) and \(T_D\), that is
\begin{align*}
    &\vert T_D (t) (u)\vert  \leq  S(t) (\vert u \vert ) ,\\
    &\vert T_D (t) (u)\vert  \leq  -S(t) (-\vert u \vert ), \\
    &\vert S(t) (u)\vert  \leq  T_N(t) (\vert u \vert ),
\end{align*}
then
\begin{align*}
    \F(u)= \int_X \frac{|\nabla u(x)|^{p(x)}}{p(x)} \; \mathrm{d}x+\int_{\partial X} B_S(x,u(x)) d\nu,
\end{align*}
where \(\nu\) is a measure on \(\partial X\) and \(B_S\) fulfills the properties stated in \cref{main theo}. Hence, \(\F\) is the energy of a realization of the \(p(x)\)-Laplacian with Robin boundary conditions. In a certain way, this implies that all {\em local} semigroups sandwiched between the semigroup generated by the \(p(x)\)-Laplacian with Neumann boundary conditions and the semigroup generated by the \(p(x)\)-Laplacian with Dirichlet boundary conditions are generated by the \(p(x)\)-Laplacian with Robin boundary conditions. This is exactly the statement proved in \cite{AW_Laplace_what_is_in_between} for the Laplace operator and in \cite{CW_p_Laplace_what_is_in_between} for the \(p\)-Laplace operator. In the recent article \cite{Ak18} (see also \cite{ArChDj21}) it was shown that at least for the case of the Laplace operator, the assumption that the sandwiched semigroup is local can be dropped; the locality is automatic and follows from positivity. It is not clear whether the assumption of locality can also be dropped in the nonlinear situation. 
\end{exmp}

In general it is not possible to define a meaningful version with Dirichlet boundary conditions for an arbitrary Dirichlet form \(\mathcal{E}\). Both functionals defined below are examples of this.

\begin{exmp}
Let \(X\) be a bounded \(C^1\) domain in \(\mathbb{R}^n\) and 
\begin{align*}
    \mathcal{E}(u)=\frac{1}{2} \int_X u^2 \; \mathrm{d}x.
\end{align*}
Obviously \(\mathcal{E}\) is a symmetric Dirichlet form, which satisfies all requirements of \cref{main theo} and generates the semigroup
\begin{align*}
    T_t(u)=e^{-t} u.
\end{align*}
In addition \(\hat{\mathcal{A}}= C(\overbar{X})\) is a core of \(\D\). But since \(\normcap\) is just the Lebesgue measure, \(\partial X\) is polar. \cref{main theo} then implies, that every order preserving semigroup \(S\) generated by a local functional \(\F\), which is pointwise exponentially decaying, that is
\begin{align*}
\vert S_t(u) \vert \leq e^{-t} \vert u \vert
\end{align*}
for every \(t \geq 0\) and \(u \in \lspace\), is generated by 
\begin{align*}
    \F(u)=\frac{1}{2} \int_X u^2 dm + \int_X B(x,u) \; \mathrm{d}\mu,
\end{align*}
where \(\mu\) is a Borel measure on \(X\) which is absolutely continuous with respect to the Lebesgue measure and \(B\) some function satisfying the properties stated in \cref{main theo}.

Even further, let us assume that, for an order preserving semigroup \(R\) which is generated by a local, lower semicontinuous, convex functional \(\mathcal{G}\),
\begin{align*}
  \vert R_t(u) \vert \leq  \vert u \vert  
\end{align*}
holds for every \(t\geq 0\) and \(u \in \lspace\). Since the semigroup \(T_t=id\) is generated by the Dirichlet form \(\mathcal{E}=0\), \cref{main theo} implies
\begin{align*}
    \mathcal{G}(u)=\int_X B(x,u) \; \mathrm{d}\mu,
\end{align*}
where, as before, \(\mu\) is a Borel measure on \(X\) which is absolutely continuous with respect to the Lebesgue measure and \(B\) is as in \cref{main theo}. Hence, semigroups generated by local functionals, which depend on derivatives, can not be pointwise decreasing. 
\end{exmp}

The examples we examined all lived on subsets of \(\mathbb{R}^n\), but this is not necessary. In fact the theorems hold for manifolds and arbitrary metric measure spaces. The main problem is to identify the space \(\hat{X}\). 

\bibliographystyle{alpha}

\end{document}